\documentclass [] {article} []

\usepackage[pdftex]{graphicx}
\usepackage{amsmath}
\usepackage{amssymb}
\usepackage{amsfonts}
\usepackage{wrapfig}
\usepackage{authblk}
\usepackage{fullpage}
\usepackage{tikz}
\usetikzlibrary{snakes}

\makeatletter
\providecommand*{\cupdot}{%
  \mathbin{%
    \mathpalette\@cupdot{}%
  }%
}
\newcommand*{\@cupdot}[2]{%
  \ooalign{%
    $\m@th#1\cup$\cr
    \sbox0{$#1\cup$}%
    \dimen@=\ht0 %
    \sbox0{$\m@th#1\cdot$}%
    \advance\dimen@ by -\ht0 %
    \dimen@=.5\dimen@
    \hidewidth\raise\dimen@\box0\hidewidth
  }%
}

\providecommand*{\bigcupdot}{%
  \mathop{%
    \vphantom{\bigcup}%
    \mathpalette\@bigcupdot{}%
  }%
}
\newcommand*{\@bigcupdot}[2]{%
  \ooalign{%
    $\m@th#1\bigcup$\cr
    \sbox0{$#1\bigcup$}%
    \dimen@=\ht0 %
    \advance\dimen@ by -\dp0 %
    \sbox0{\scalebox{2}{$\m@th#1\cdot$}}%
    \advance\dimen@ by -\ht0 %
    \dimen@=.5\dimen@
    \hidewidth\raise\dimen@\box0\hidewidth
  }%
}
\makeatother

\newenvironment{proof}{\noindent\textbf{Proof}\\}{\noindent$\Box$\\}

\newtheorem{lemma}{Lemma}
\newtheorem{theorem}[lemma]{Theorem}

\newtheorem{example}[lemma]{Example}
\newtheorem{dfn}[lemma]{Definition}

\newtheorem{prop}[lemma]{Proposition}

\author{Robert Barham}
\title{Notes on Representing $\aleph_{0}$-categorical Linear Orders}
\begin{document}
\maketitle

\begin{abstract}
These notes find a canonical representation of the $\aleph_0$-categorical linear orders based on Joseph Rosenstein's description.  A unique minimal representation, called the normal form, is obtained.
\end{abstract}

The $\aleph_{0}$-categorical linear orders were classified by Joseph Rosenstein in \cite{JGR2}, where he constructs them from finite linear orders using two operations.

\begin{dfn}
If $\langle L_0 , <_0 \rangle$ and $\langle L_1 , <_1 \rangle$ are linear orders then their \emph{concatenation}, denoted by $L_0 \,^{\wedge} L_1$ is the linear order $\langle L_0 \cup L_1 , < \rangle$, where
$$x< y \quad \mathrm{iff} \quad \left\lbrace
\begin{array}{r c c r}
(x,y \in L_0 & \mathrm{and} & x <_0 y) & \quad\mathrm{or} \\
(x,y \in L_1 & \mathrm{and} & x <_1 y) & \quad\mathrm{or} \\
(x \in L_0 & \mathrm{and} & y \in L_1)
\end{array} \right. $$
\end{dfn}

\begin{dfn}
$\langle \mathbb{Q}_n, <_{\mathbb{Q}_n}, C_1 \ldots C_n \rangle$ is the Fra\"{i}ss\'{e} generic $n$-coloured partial order, i.e. the countable dense linear order with $n$ colours which occur interdensely (for all $x$ and $y$ there are $z_1, \ldots z_n$ between $x$ and $y$ such that $C_i(z_i)$ holds for each $i$).
\end{dfn}

$\mathbb{Q}_n$ is the Fra\"iss\'e limit of $n$-coloured linear orders, and hence is $\aleph_0$-categorical.

\begin{dfn}
Let $\langle L_1 , <_1 \rangle, \ldots, \langle L_n , <_n \rangle$ be linear orders.  For each $q \in \mathbb{Q}_n$ we define $L(q)$ to be a copy of $\langle L_i , <_i \rangle$  	if $\mathbb{Q}_n \models C_i(q)$.  The $\mathbb{Q}_n$-\emph{shuffle} of $\langle L_1 , <_1 \rangle, \ldots, \langle L_n , <_n \rangle$, denoted by $\mathbb{Q}_n(L_1, \ldots L_n)$, is the linear order $\langle \bigcup_{q \in \mathbb{Q}_n} L(q), < \rangle $, where
$$x< y \quad \mathrm{iff} \quad \left\lbrace
\begin{array}{c c c c}
x,y \in L(q) & \mathrm{and} & x <_i y & \mathrm{or}\\
x \in L(q) \, , \, y \in L(p) & \mathrm{and} & q <_{\mathbb{Q}_n} p
\end{array} \right. $$
\end{dfn}

\begin{theorem} [Rosenstein]  $L$ is an $\aleph_{0}$-categorical linear order if and only if $L$ can be constructed from singletons by a finite number of concatenations or shuffles.\cite{JGR1}
\end{theorem}

This leads to a formal representation of the $\aleph_{0}$-categorical linear orders

\begin{dfn}
A \emph{term} is built as follows:
\begin{description}
\item[Singleton] The singleton 1 is a term
\item[Concatenation] If $t_0, t_1$ are terms then $t_0 \, ^\wedge \, t_1$ is a term.
\item[$\mathbb{Q}_n$-shuffle] If $t_0, \ldots t_{n-1}$ are terms then $\mathbb{Q}_n (t_0, \ldots, t_{n-1})$ is a term.
\end{description}
How terms represent linear orders is obvious.  If $t$ is a term, then let $L_t$ be the linear order represented by $t$.  We say that a term $t$ is a shuffle if there are $t_0, \ldots t_{n-1}$ such that $t= \mathbb{Q}_n(t_0, \ldots t_{n-1})$.
\end{dfn}

Concatenation is obviously associative, so we will not bother using brackets.  The terms can be interpreted as linear orders in the obvious way.  Note that multiple terms can denote the same linear order.

\begin{example}
The rationals can be represented by both $\mathbb{Q}_1(1)$ and $\mathbb{Q}_2(1,1)$.
\end{example}

This leads to the question of which terms represent the same $\aleph_{0}$-categorical linear order, and if there is a canonical choice of representative.  We also discuss a slightly wider class of linear orders by allowing infinite concatenation, which will be dealt with by discussing infinite sequences of terms.  While their concatenation will not necessarily be an $\aleph_{0}$-categorical linear order we will be able to arrive at a canonical representation of these as well.

Lemma \ref{lemma:nfrep} lists the ways in which terms can represent the same linear order.  These all concern terms which encode infinite linear orders, as the representation for finite linear orders is automatically unique.

\begin{dfn}
The \textbf{complexity} of a term $T$, written as $c(T)$, is the number of concatenations and shuffles in the term added to the sum of all the lengths of the shuffles contained in the term.  The complexity of a sequence of terms is the sum of the complexities of all of the terms that appear in the sequence and the length of the sequence.
\end{dfn}

Note that this means that all infinite sequences have infinite complexity.

\begin{dfn}The \textbf{depth} of a term $T$, written as $d(T)$, is defined as follows:
\begin{itemize}
\item If $T$ represents a finite linear order then $d(T)=0$.
\item If $T$ is of the form $\mathbb{Q}_n (t_0, \ldots, t_{n-1})$ then $d(T)=max(d(t_i))+1$.
\item If $T$ is of the form $s_0 \,^\wedge \ldots \,^\wedge s_{n-1}$ then $d(T)=max(d(s_j))$.
\end{itemize}
\end{dfn}

\begin{lemma}\label{lemma:nfrep} \label{lemma:nfperm} \label{lemma:nfnest} \label{lemma:nfconc}
Let $t_0, \ldots, t_{n-1}$ be terms, let $m \leq n$ and let $f$ be a permutation of $n$.  We also let $\tau$ be either the empty set or one of the $t_i$. Then the following are isomorphic to $\mathbb{Q}_n(t_0, t_1, \ldots ,t_{n-1})$:
\begin{enumerate}
\item $\mathbb{Q}_n(t_{f(0)}, t_{f(1)}, \ldots ,t_{f(n-1)})$;
\item $\mathbb{Q}_{n+1} (t_0, \ldots, t_{n-1},t_m)$;
\item $\mathbb{Q}_{m+1}(t_0, \ldots t_{m-1}, \tau_0 \,^\wedge \mathbb{Q}_n(t_0, \ldots t_{n-1})\,^\wedge \tau_1)$ where $\tau_0,\tau_1 \in \lbrace \emptyset, t_0, \ldots , t_{n-1} \rbrace$; and
\item $\mathbb{Q}_{m}(t_0, \ldots ,t_{m-1}) ^\wedge \tau ^\wedge \mathbb{Q}_{m}(t_0, \ldots ,t_{m-1})$ where $\tau \in \lbrace \emptyset, t_0, \ldots , t_{m-1} \rbrace$.
\end{enumerate}
\end{lemma}
\begin{proof}
The roles of the colours in $\mathbb{Q}_n$ are interchangeable, so
$$\mathbb{Q}_n(t_0, \ldots t_{n-1}) \cong \mathbb{Q}_n(t_{f(0)}, t_{f(1)}, \ldots t_{f(n-1)})  $$

The structure $\langle \mathbb{Q}_{n+1}, <_{\mathbb{Q}_n}, C_1, \ldots, C_{m-1}, C_m \vee C_{n+1}, C_{m+1} \ldots, C_n \rangle$ is a countable dense linear order with $n$ colours which occur interdensely, and therefore
$$\langle \mathbb{Q}_{n+1}, <, C_0, \ldots, C_{m-1}, C_m \vee C_{n}, C_{m+1} \ldots, C_{n-1} \rangle \cong \langle \mathbb{Q}_{n}, <, C_0, \ldots,  \ldots, C_{n-1} \rangle$$
and $\mathbb{Q}_{n+1} (t_0, \ldots, t_{n-1},t_m) \cong \mathbb{Q}_n(t_0, t_1, \ldots t_{n-1}) $.

The structure $\mathbb{Q}_{m+1}(t_0, \ldots t_{m-1}, \tau_0 \,^\wedge \mathbb{Q}_n(t_0, \ldots t_{n-1})\,^\wedge \tau_1)$ is obtained by replacing the $C_i$ coloured elements of $\mathbb{Q}_{m+1}$ by $t_i$ if $i<m-1$ or by $\tau_0 \,^\wedge \mathbb{Q}_n(t_0, \ldots t_{n-1})\,^\wedge \tau_1$ if $i=m$.  The structure $\mathbb{Q}_{n}(t_0, \ldots t_{n-1})$ is obtained by replacing the $C_i$ coloured elements of $\mathbb{Q}_{n}$ by $t_i$.  Let $M$ be the coloured linear order obtained by replacing the $C_{m}$ coloured elements of $\mathbb{Q}_{m+1}$ by $x \,^\wedge \mathbb{Q}_n \,^\wedge y$, where $x$ and $y$ are coloured according to the values taken by $\tau , \sigma \in \lbrace \emptyset, t_0, \ldots, t_{n-1} \rbrace$.  (If $\tau=\emptyset$ then we delete $x$.)  We may also obtain $\mathbb{Q}_{n}(t_0, \ldots t_{n-1})$ by replacing the $C_i$ coloured elements of $M$ by $t_i$.  This $M$ is a dense linear order in which the colours $C_0, \ldots C_{n-1}$ occur interdensely, and so $M \cong \mathbb{Q}_n$ and hence
$$\mathbb{Q}_n(t_0, \ldots t_{n-1}) \cong \mathbb{Q}_{m+1}(t_0, \ldots t_{m-1}, \tau_0 \,^\wedge \mathbb{Q}_n(t_0, \ldots t_{n-1})\,^\wedge \tau_1)$$

Let $M$ be the structure $\mathbb{Q}_{n} \,^\wedge \lbrace x \rbrace \,^ \wedge \mathbb{Q}_n$ where $x$ is coloured by $C_i$ if and only if $\tau=t_i$.  The structure $\mathbb{Q}_{n}(t_0, \ldots ,t_{n-1}) ^\wedge \tau ^\wedge \mathbb{Q}_{n}(t_0, \ldots ,t_{n-1})$ can obtained by replacing the $C_i$ coloured elements of $M$ by $t_i$, however $M$ is a dense coloured linear order where the colours $C_i$ for $i<n$ occur interdensely, and so $M \cong \mathbb{Q}_n$.  Therefore
$$\mathbb{Q}_n(t_0, \ldots t_{n-1}) \cong \mathbb{Q}_{n}(t_0, \ldots ,t_{n-1}) ^\wedge \tau ^\wedge \mathbb{Q}_{n}(t_0, \ldots ,t_{n-1}))$$
\end{proof}

\begin{dfn}\label{dfn:nf}
We use induction over the formation of terms to define when a term is in \textbf{normal form} (n.f.).
\begin{enumerate}
\item All finite terms are in n.f..
\item A term of the form $\mathbb{Q}_{m}(t_0, \ldots ,t_{m-1})$ is in n.f. if:
\begin{enumerate}
\item all the $t_i$ are in n.f.; and
\item Numbers 2 and 3 of Lemma \ref{lemma:nfrep} do not apply.
\end{enumerate}
If the $t_i$ are permuted then the term is unaltered.
\item A term of the form $t_0 \,^\wedge \ldots \,^\wedge t_{n-1}$ is in n.f. if all the $t_i$ are in n.f. and no $t_{i-1} \,^\wedge t_i \,^\wedge t_{i+1}$ or $t_{i} \,^\wedge \emptyset \,^\wedge t_{i+1}$ satisfy Number 4 of Lemma \ref{lemma:nfconc}.
\end{enumerate}
A possibly infinite sequence of terms $(s_i)$ is said to be in \textbf{normal form} if:
\begin{enumerate}
\item each $s_i$ is in normal form;
\item no $s_{i-1} \,^\wedge s_i \,^\wedge s_{i+1}$ or $s_{i} \,^\wedge \emptyset \,^\wedge s_{i+1}$ satisfy Number 4 of Lemma \ref{lemma:nfconc};
\item if $s_j$ is finite either:
\begin{enumerate}
\item $s_{j+1}$ is infinite; or
\item $(s_i)$ is an infinite sequence and $s_j = s_k = 1$ for all $k \geq j$.
\end{enumerate}
\end{enumerate}
If $(s_i)$ is a sequence in normal form, and $L$ is a linear order represented by $(s_i)$ then we say that $(s_i)$ is the \emph{n.f. representation} of $L$.
\end{dfn}

\begin{dfn}
Let $t$ be a shuffle, and let $s$ be a term shuffled by $t$.  The we say that $\sigma \in\alpha(s,t)$ if $\sigma \subseteq L_t$ is obtained in the construction of $L_t$ by replacing an element of $\mathbb{Q}_n$ by $L_s$.
\end{dfn}

\begin{lemma}\label{lemma:nottermtoterm}
Let $S$ and $T$ be shuffles such that $L_S \cong L_T$, witnessed by $\phi$.  Suppose $S$ shuffles an $s$ such that for all $t$ shuffled by $T$ and all $\sigma \in \alpha(s,S)$
$$\phi(\sigma) \not\in \alpha(t,T)$$
Then $L_s$ is isomorphic to $A_0 \,^\wedge L_S \,^\wedge A_1$, where $A_0$ (resp. $A_1$) is either equal to a terminal (resp. initial) segment one of the terms that $T$ shuffles or empty.
\end{lemma}
\begin{proof}
$L_T$ is obtained by replacing the points of $\mathbb{Q}_n$ with the appropriate linear order.  Let $\chi$ be the map from $T$ to $\mathbb{Q}_n$ that sends a point in $L_T$ to the point in $\mathbb{Q}_n$ that $t$ was obtained from when building $L_T$.

Since $\sigma$ is a bounded convex subset of $L_S$ and these properties are preserved by the maps $\phi$ and $\chi$ we know that $\chi(\phi(\sigma))$ is a bounded interval.  We denote the interior of $\chi(\phi(\tilde{s}))$ by $I$.

$\chi^{-1}(I)$ is contained in $\phi(\sigma)$ and since open intervals of $\mathbb{Q}_n$ are isomorphic to $\mathbb{Q}_n$ we have that $\chi^{-1}(I)$ is isomorphic to $L_T$.  We now define two sets
\begin{itemize}
\item $A_0 := \lbrace x \in \phi(\tilde{s}) \, : \, x < \chi^{-1}(I) \rbrace$
\item $A_1 := \lbrace x \in \phi(\tilde{s}) \, : \, x > \chi^{-1}(I) \rbrace$
\end{itemize}
$A_0$ must be contained in a copy of some term that $T$ shuffles (say $t_0$), as otherwise there would be an $x \in A_0$ such that $\chi(x) \in I$.  Since $\phi(\sigma)$ is convex, $A_0$ must be a terminal segment of some $\tau_0 \in \alpha(t_0,T)$.  Similarly $A_1$ must be an initial segment some $\tau_1 \in \alpha(t_1,T)$.  It is here that we note that $\phi(\sigma) \cong A_0 \,^\wedge \chi^{-1}(I) \,^\wedge A_1$, as required
\end{proof}

\begin{prop}
Let $S$ and $T$ be terms in normal form.  If $L_S \cong L_T$ then $S=T$.
\end{prop}
\begin{proof}
Let $\phi:L_S \rightarrow L_T$ be an isomorphism.  We assume that $d(T) \leq d(S)$.  If $L_S$ and $L_T$ are finite then they are automatically represented by the same term, so we now assume that $L_S$ and $L_T$ are infinite.  Let $T = \mathbb{Q}_n (t_0, \ldots, t_{n-1})$ and $S = \mathbb{Q}_m (s_0, \ldots, s_{m-1})$.

We will prove this by induction on the depth of the shuffle of $T$, with both stages of the induction also proved by inducting on the depth of the shuffle of $S$.  Let $\phi$ be an isomorphism from $L_S$ to $L_T$.  Note that in the base case $d(T)=1$ means that every member of each $\alpha(t_i,T)$ is finite, and so can be recognised as the maximal finite convex subsets of $L_T$.

\paragraph*{$d(T)=1$ and $d(S)=1$.}
The base case for the induction on the depth of the shuffle of $T$ will be proved by another another induction, this time on the depth of $S$.  Let $d(S)=1$, i.e. each $s_j$ is finite.  For every $\tau \in \alpha(t_i,T)$ the image $\phi(\tau) \in \alpha(s_j,S)$ for some $s_j$, and for every $\sigma \in \alpha(s_j,S)$ the preimage $\phi^{-1}(\sigma)\in \alpha(t_i,T)$ for some $t_i$.  This means that for every $t_i$ there is an $s_j$ such that $t_i=s_j$ and vice versa.  Since $T$ and $S$ are in normal form both $m$ and $n$ are minimal.  This means that $n=m$ and  $(t_0, \ldots t_{n-1})$ is a permutation of $(s_0, \ldots s_{m-1})$ and $T=S$.

\paragraph*{$d(T)=1$ and $d(S)=p$ where $p>1$.}
By the induction hypothesis if $U$ is a n.f. term such that $d(U)<p$ and if $L_U$ is isomorphic to $L_T$ then $T=U$.  Since $d(S)>1$ at least one of the $s_i$ must be infinite.  Let $s_j$ be one of the infinite $s_i$.  Since $s_j$ is infinite, $\phi^{-1}(\sigma)\not\in \alpha(t_i,T)$ for all $\sigma \in \alpha(s_j,S)$.  Lemma \ref{lemma:nottermtoterm} shows that $\tilde{s_j} \cong A_0 \,^\wedge L_T \,^\wedge A_1$.

If $A_1$ is a proper terminal segment of $\tau$ then there are $x,y \in \tau$ such that $\phi^{-1}(x)$ is not contained in some $\sigma \in \alpha(s_j,S)$, but $\phi^{-1}(y)$ is.  This implies that the interval $[x,y]$ is finite while the interval $[\phi^{-1}(x),\phi^{-1}(y)]$ is infinite, contradicting the fact that $\phi$ is an isomorphism, so $\tau \in \lbrace \emptyset \rbrace \cup \alpha(t_0,T)$.  Similarly $A_2$ is either empty or contained in $\alpha(t_1,T)$.  From this we can conclude that $s_j = t_0 \,^\wedge s_j' \,^\wedge t_1$ where $L_{s_j'} \cong L_T$.

Since $d(s_j')<p$ we know that $s_j'=T$.  This means that every $s_j$ is either equal to a $t_i$ or is of the form $\sigma \,^\wedge T \,^\wedge \tau$ contradicting that both $S$ and $T$ are in normal form.

\paragraph*{$d(T)=p$ and $d(S)=p$ where $p>1$.}
Now suppose that if $U$ and $U'$ are in n.f. and $d(U)<p$ then $L_U \cong L_{U'} \Rightarrow U=U'$.  Pick an $s_j$, and let $\sigma \in \alpha(s_j,S)$.  We know that $\phi(\sigma)$ must be entirely contained in one of the $\tilde{t_i}$, as otherwise we would be able to apply Lemma \ref{lemma:nottermtoterm} and find that $\sigma$ is isomorphic to $A_0 \,^\wedge L_T \,^\wedge A_1$, which implies that $L_T$ is isomorphic to something of depth less than $p$, giving a contradiction.

So $\phi(\sigma)$ is entirely contained some $\tau \in \alpha(t_i,T)$.  Similarly $\phi^{-1}(\tau)$ must be entirely contained in an element of $\alpha(s_j,S)$.  Hence $\phi(\sigma)=\tau$ and $t_i=s_j$.  Therefore every term shuffled by $S$ is equal to a term shuffled by $T$ and vice versa.  Since both $S$ and $T$ are in normal form there can be no repeated terms in the shuffle, so $L_T=L_S$.

\paragraph*{$d(T)=p$ and $d(S)=r$ where $r>s$.}
Now suppose that $L_S$ has depth $r$ where $p<r$, and for all $U$ and $U'$ in n.f. if $d(U)<r$ then $L_U \cong L_{U'} \Rightarrow U=U'$.  There is a $\sigma \in \alpha(s_j,S)$ such that $\phi(\sigma)$ is not entirely contained in any element of any $\alpha(t_i,T)$ in which case $\sigma \cong A_0 \,^\wedge L_T \,^\wedge A_1$.  The depth of every $t_i$ is less than $r$ and so $t_i= \sigma \,^\wedge T \,^\wedge \tau$.

Suppose $A_0$ is a proper non-empty subset of an element of $\alpha(t_0,T)$.  This implies that $\phi^{-1}(t_0)$ is not contained entirely in one of the terms shuffled by $S$ and so by the Lemma \ref{lemma:nottermtoterm}, $t_0 \cong B_0 \,^\wedge L_S \,^\wedge B_1$.  As in the previous step this results in a term with depth less than $p$ being isomorphic to $L_S$, contradicting the induction hypothesis.  This means that $A_0$ is either empty or equal to $t_0$.  Similarly $A_1$ is either empty or equal to $t_1$.  Since $s_j$ has depth less than $r$ it must be equal to $t_0 \,^\wedge T \,^\wedge t_1$.  Therefore every term shuffled by $S$ is either a term shuffled by $T$, or equal to $t_0 \,^\wedge T \,^\wedge t_1$, and hence $S$ is not in normal form.
\end{proof}

\begin{prop}
If $(T_i)$ and $(S_i)$ are sequences in normal form whose linear orders encoded by their concatenations are isomorphic then $T_j=S_j$ for all $j$.
\end{prop}
\begin{proof}
Let $L_i$ and $M_i$ be the linear orders encoded by $T_i$ and $S_i$ respectively and let $L$ and $M$ correspond to the concatenations.  Let $ \phi : L \rightarrow M$ be an isomorphism.  If $T_0$ is finite then $\phi$ must map $L_0$ to $M_0$ as automorphisms preserve least elements.  If $L_0$ is infinite suppose that $i$ is the greatest number such that an element of $L_0$ is mapped to $M_i$.

If $i=1$ then $M_1$ cannot be finite as $L_0$ does not have a maximal element.  Since neither $M_0$ nor $M_1$ have maximal or minimal elements $\phi^{-1}(M_0)$ and $\phi^{-1}(M_1)$ do not have maximal or minimal elements, hence they are unbounded open intervals of $L_0$ and isomorphic to $L_0$.  However $(M_i)$ is in normal form, so $M_0$ and $M_1$ cannot be isomorphic.

If $i>1$ then we consider $\phi^{-1}(M_0) \,^\wedge \ldots \,^\wedge \phi^{-1}(M_i)$.  The argument of the previous paragraph can be adapted to show that $M_0$ and $M_i \cong L_0$.  Since $L_0$ is a shuffle, $M_0$ does not have a maximal element and $M_i$ does not have a minimal element.  We may assume that for $0 < j< i$ none of the $M_j$'s are isomorphic to $L_0$, as otherwise we can consider $M_0, \ldots, M_j$ instead of $M_0, \ldots, M_i$.

Since none of the $M_j$ are isomorphic to $L_0$ we know that $\phi^{-1}(M_1 \,^\wedge \ldots M_{i-1}) $ is contained in a copy of one of the terms that $L_0$ shuffles.  Moreover the $\phi$-preimage of $(M_1 \,^\wedge \ldots \,^\wedge M_{i-1}) $ must be a copy of one of the terms that $L_0$ shuffles as $\phi$ is an isomorphism and being strictly contained in a copy of a term would prevent $\phi$ from being a surjection.  This means that we are able to apply Lemma \ref{lemma:nfrep} to $M_0 \,^\wedge \ldots \,^\wedge M_i $, contradicting the assumption that $(M_i)$ is in normal form.

From this we deduce that $L_0 \cong M_0$, and repeating this argument shows that $L_j \cong M_j$ for all $j$.  Therefore any automorphism between $L$ and $M$ must send $L_j$ to $M_j$.  Since $(L_i)$ and $(M_j)$ are both in normal form, the singletons $L_j$ and $M_j$ are also in normal form and are represented by the same term by the preceding lemma, proving that the sequences $(T_i)$ and $(S_i)$ are equal.
\end{proof}

So now we have that if a term or a sequence of terms has a normal form representation then this representation is unique and that different normal form sequences represent different linear orders.  Finally we need to show that every sequence of terms has a normal form sequence that encodes the same linear order.  We shall first show that all finite sequences of arbitrary length have a normal form representation before considering infinite sequences.

\begin{prop}
If $t$ is a term that is not in normal form then there is a term $s$ such that $L_t=L_s$ and $s$ is in normal form.  Furthermore, $c(s)< c(t)$.
\end{prop}
\begin{proof}
We prove this by induction on the complexity of terms.  All finite terms are in normal form, so the base case is immediate.  Suppose that if $c(s)<n$ then there is an $s'$ such that $L_s'=L_s$ and $s'$ is in normal form, and $c(s')<c(s)$.  Let $t$ be such that $c(t)=n$ and $t$ is not in normal form.

$t$ is not finite, as $t$ is not in normal form.

Suppose that $t= \mathbb{Q}_n(t_0, \ldots, t_{n-1})$.  If $t_i$ is not in normal form, then let $t'_i$ be such that $L_{t_i}=L_{t_i'}$.  Let $t':= \mathbb{Q}_n(t_0, \ldots, t_i', \ldots, t_{n-1})$.  Since $c(t_i')<c(t_i)$, we know that $c(t')<c(t)$, and so $c(t')$ is either in normal form, in which case we are done, or by the induction hypothesis there is a $t''$ which is in normal form, $c(t'')<c(t)$ and $L_{t''}=L_{t}$.

Suppose that $t= t_0 \,^\wedge \ldots \,^\wedge t_{n-1}$.  If there is an $i$ such that $t_i \,^\wedge t_{i+1} \,^\wedge t_{i+2}$ represents the same linear order as $t_i$ then we can replace $t_i \,^\wedge t_{i+1} \,^\wedge t_{i+2}$ by $t_i$ to obtain a term $t'$ such that $L_t=L_{t'}$ and $c(t')<c(t)$.  Similarly, if there is an $i$ such that $t_i \,^\wedge \emptyset \,^\wedge t_{i+1}$ represents the same linear order as $t_i$ then we can replace $t_i \,^\wedge \emptyset\,^\wedge t_{i+1}$ by $t_i$ to obtain a term $t'$ such that $L_t=L_{t'}$ and $c(t')<c(t)$.
\end{proof}

\begin{prop}
If $(t_i)$ is an infinite sequence of terms there is a sequence $(s_i)$ such that $(s_i)$ represents the same linear order as $(t_i)$ and is in normal form.
\end{prop}
\begin{proof}
If we are given a sequence of terms $(t_i)$ we can construct a normal form sequence $(s_j)$ which represents the same linear order inductively.  We first let $s^0_0$ be the normal form representative of $t_0$.  Suppose that we have considered $t_i$ for $i<n$, obtaining the sequence $(s^{n-1}_j)$ for $j<m$.  Let $t_n'$ be the normal form term representing the same linear order as $t_n$.  We construct $(s_j)$ as follows:
\begin{enumerate}

\item If both $s^{n-1}_{m-1}$ and $t_n'$ are finite, but there is some $t_k$ which is a $\mathbb{Q}_n$-shuffle for $k>n$, then let $s^{n}_{m-1}=s^n_{m-1} \,^\wedge t_n'$ and $s^n_{i}=s^{n-1}_i$ for $i<n$.  We then consider $t_{n+1}$.
\item If both $s^{n-1}_{m-1}$ and $t_n'$ are finite, and there is no $t_k$ which is a $\mathbb{Q}_n$-shuffle for $k>n$, then if $j<m-1$ let $s^\omega_i=s^{n-1}_i$.  Otherwise, let $s^\omega_i=1$ for all $i\geq m$ and stop.
\item If $s_{m-2} \,^\wedge s_{m-1} \,^\wedge t_n'$ satisfies Number 4 of Lemma \ref{lemma:nfrep}, then let $s^{n}_{m-1}= \emptyset$ and $s^n_{i}=s^{n-1}_i$ for $i<m$.  We then consider $t_{n+1}$.
\item If $s_{m-1} \,^\wedge \emptyset \,^\wedge t_n'$ satisfies Number 4 of Lemma \ref{lemma:nfrep}, then let $s^n_{i}=s^{n-1}_i$ for $i<m$.  We then consider $t_{n+1}$.
\item Otherwise let $s^{n}_{m}=t_n'$ and $s^n_{i}=s^{n-1}_i$ for $i<m$.  We then consider $t_{n+1}$.

\end{enumerate}

It is easy to see that each sequence $(s^n_i)$ represents the same linear order as $(t_i)_{i=0}^n$.  If this process never terminates, let $s^\omega_i$ be the term taken by the tail of the sequence $(s^n_i)_{n \in \mathbb{N}}$.  If that sequence has no constant tail then let $s^\omega_i=\emptyset$.  By construction, the sequence $(s^\omega_i)$ is in normal form.

If we obtain $(s^\omega_i)$ via Number 2 then $(s^\omega_i)$ represents the same linear order as $(t_i)$.  Suppose that we obtain $(s^\omega_i)$ as a limit.  Suppose that $s^\omega_{j+1} \not= \emptyset$.  Let $m_j$ be the least number such that $s^n_{j+1} \not= \emptyset$ for all $m>n$.  Then $(s^\omega_i)_{i=0}^{j}$ represents the same linear order as $(t_i)_{i=0}^{m_j-1}$.  Therefore $s^\omega_j$ represents the same linear order as $(t_i)_{m_{(j-1)}}^{m_j - 1}$, and if $(s^\omega_i)$ is an infinite sequence then it represents the same linear order as $(t_i)$.

Suppose that $(s^\omega_i)$ terminates, and let $s^\omega_j$ be the final element.  Then eventually we always eventually apply Number 3 to $s^n_j \,^\wedge s^n_{j+1} \,^\wedge t_{n}'$ or Number  4 to $s^n_j \,^\wedge t_{n}'$ for some $n$.  Therefore $s^\omega_j$ is isomorphic to the linear order represented by the appropriate tail of $(t_i)$, and $(s^\omega_i)$ represents the same linear order as $(t_i)$.
\end{proof}

\end{document}